\documentclass[a4paper, 12pt]{article}
\usepackage[left=1in, right=1in]{geometry}
\usepackage{amsmath, amssymb, amsthm}
\usepackage{mathrsfs}
\usepackage{mathtools}
\usepackage{tikz-cd}
\usepackage[shortlabels,inline]{enumitem}
\usepackage{float}
\usepackage{hyperref}
\usepackage[abbrev]{amsrefs}
\usepackage{authblk}
\usepackage{lipsum}
\newcommand\blfootnote[1]{%
  \begingroup
  \renewcommand\thefootnote{}\footnote{#1}%
  \addtocounter{footnote}{-1}%
  \endgroup
}

\newcommand{\dd}{\mathrm{d}}
\newcommand{\id}{\mathrm{id}}
\DeclareMathOperator{\supp}{supp}
\DeclarePairedDelimiter{\abs}{\lvert}{\rvert}
\DeclarePairedDelimiter{\norm}{\lVert}{\rVert}

\newtheorem{definition}{Definition}
\newtheorem{example}{Example}
\newtheorem{theorem}{Theorem}
\newtheorem{lemma}{Lemma}
\newtheorem{proposition}{Proposition}
\newtheorem{remark}{Remark}
\newtheorem{corollary}{Corollary}

\begin{document}

\title{Multiple Non-radial Solutions for Coupled Schr\"{o}dinger Equations}
\author[1]{Xiaopeng Huang}
\author[2]{Haoyu Li}
\author[3*]{Zhi-Qiang Wang}
\affil[1]{College of Mathematics and Statistics,
	Fujian Normal University,
	Fuzhou, Fujian 350117, P. R. China
}
\affil[2]{Departamento de Matem\'atica,
	Universidade Federal de S\~{a}o Carlos,
	S\~{a}o Carlos-SP, 13565-905, Brazil
}
\affil[3]{Department of Mathematics and Statistics,
 	Utah State University,
 	Logan, Utah 84322, USA
}

%
%

\maketitle
\blfootnote{*Corresponding author. E-mail: zhi-qiang.wang@usu.edu
}
\begin{abstract}
	The paper deals with the existence of non-radial solutions for an $N$-coupled nonlinear elliptic system. In the repulsive regime with some structure conditions on the coupling and for each symmetric subspace of rotation symmetry, we prove the existence of an infinite sequence of non-radial positive solutions and an infinite sequence of non-radial nodal solutions.
\end{abstract}



\section{Introduction}
In this paper, we consider the following $N$-coupled system of nonlinear elliptic equations
\begin{equation}\label{eq:main}
	\begin{cases}
		-\Delta u_j + \lambda_j u_j = \mu_j u_j^3 +\sum_{k\neq j} \beta_{jk} u_ju_k^2, &\text{in $\Omega$,}\\
		u_j=0, &\text{on $\partial \Omega$ \qquad }
	\end{cases}
	\text{for each $j=1,\dots,N$.}
\end{equation}
This class of systems arises naturally when seeking standing wave solutions to the time-dependent Schr\"odinger system~\eqref{e:Main} which models many physical problems
\begin{equation}\label{e:Main}
    \begin{cases}
		\mathrm{i}\frac{\partial}{\partial t}\Phi_j+\Delta\Phi_j+\mu_j|\Phi_j|^2\Phi_j+\sum_{i\neq j}\beta_{ij}|\Phi_i|^2\Phi_j=0 &   \text{for $t>1$, $x\in\mathbb{R}^n$,}\\
     \Phi_j(x,0)=\Phi_{j0}(x), & j=1,\dots,N.
	\end{cases}
\end{equation}
Such equations are found in Kerr-like photorefractive media \cite{akhmedievPartiallyCoherentSolitons1999}, the Hartree-Fock theory for Bose-Einstein condensates \cite{esryHartreeFockTheoryDouble1997}, and other physical phenomena. Solutions to the system \eqref{eq:main} can also describe the steady states of the distribution of different species of some diffusion systems.
In physical terms, when the coupling constant $\beta_{ij}$ ($i\neq j$) is positive, it is referred as attractive, and when the coupling constant $\beta_{ij}$ is negative, it is said to be repulsive, leading to very different behaviors of solutions for different coupling regimes.

There have been extensive studies in the last twenty years or so on these systems and related systems to reveal a rich set of interesting and important phenomena. See the seminal paper of Lin and Wei \cite{linGroundStateCoupled2005}, the subsequent works \cites{ambrosettiStandingWavesCoupled2007,bartschNoteGroundStates2006,bartschLiouvilleTheoremApriori2010,dancerPrioriBoundsMultiple2010,liMultipleNodalSolutions2021,liuMultipleBoundStates2008,liuGroundStatesBound2010,liuVectorSolutionsPrescribed2019,norisUniformHolderBounds2010,sirakovLeastEnergySolitary2007,terraciniMultipulsePhasesKmixtures2009,weiNonradialSymmetricBound2007,weiRadialSolutionsPhase2008,liuMultipleMixedStates2015} and references therein.
In particular, for the repulsive cases, there have been works on the existence of multiplicity results of segregation-type solutions, symmetry breaking of solutions, and phase-separations among other things. One comment feature for this class of systems is the existence of many so-called semi-trivial solutions (solutions being non-zero as vectors but containing at least one zero component).
In terms of the classification of solutions, Liu and Wang \cites{liuMultipleBoundStates2008, liuGroundStatesBound2010} gave the existence of infinitely many non-trivial solutions (solutions with every component non-zero) regardless of the existence of how many semi-trivial solutions.
\cites{weiRadialSolutionsPhase2008,dancerPrioriBoundsMultiple2010,bartschLiouvilleTheoremApriori2010} give the existence of an infinite sequence of positive solutions, further showing disparity of qualitative properties solutions with the counterpart of the scale field equation $-\Delta w+w=w^3$ for which uniqueness of positive solutions is well known (\cites{gidasSymmetryRelatedProperties1979,kwongUniquenessPositiveSolutions1989}).
All these works in \cites{weiRadialSolutionsPhase2008,dancerPrioriBoundsMultiple2010,bartschLiouvilleTheoremApriori2010} explore some symmetry structure of the coupling matrix $\mathcal{B}= (\beta_{ij})$ when writing $\beta_{jj}=\mu_j$ and
the work of \cites{weiRadialSolutionsPhase2008,dancerPrioriBoundsMultiple2010} was done by using variational methods and that of \cite{bartschLiouvilleTheoremApriori2010} by a global bifurcation approach.
The work \cites{weiRadialSolutionsPhase2008,dancerPrioriBoundsMultiple2010} was further extended to the $N$-system for any $N\geq 2$ in \cite{tianMultipleSolitaryWave2011} (see also related work in \cite{terraciniMultipulsePhasesKmixtures2009}).
When the domain is radially symmetric (including the case of the entire space $\mathbb R^n$), the above results of \cites{weiRadialSolutionsPhase2008,bartschLiouvilleTheoremApriori2010,dancerPrioriBoundsMultiple2010}
also give the existence of an infinite sequence of radial positive solutions.
A natural question is to study non-radial solutions in the setting of radially symmetric domains for which little work has been done so far.
Some numeric work was done in \cite{changSegregatedNodalDomains2004} indicating the existence of a rich variety of different types of non-radial solutions.
In \cite{weiNonradialSymmetricBound2007}, Wei and Weth constructed non-radial ground state solutions in some symmetric subspaces and the bifurcation results in \cite{bartschLiouvilleTheoremApriori2010}
also pointed out various type of non-radial positive solutions.
Motivated by these works, one major goal of our current paper is to investigate the existence of multiple non-radial positive solutions. More precisely, for bounded radially symmetric domains, within each symmetric subspace we construct an infinite sequence of non-radial positive solutions.
Furthermore, our ideas can also be adapted partially to the study of non-radial nodal (sign-changing) solutions.
For nodal solutions, there have been quite some works in the literature such as \cite{satoMultipleExistenceSemipositive2013}, \cite{chenInfinitelyManySignchanging2016}, \cite{liuMultipleMixedStates2015}, \cite{tavaresSignchangingSolutionsCompetitiondiffusion2012}, \cite{liMultipleNodalSolutions2021}, etc.
In \cite{liuMultipleMixedStates2015}, the existence of an infinite sequence of nodal solutions was constructed by using minimax method in the presence of invariant sets of a negative pseudo gradient flow, which gives a sequence of radial nodal solutions when the domain is radially symmetric.
Further classification was done in \cite{liMultipleNodalSolutions2021} in which with prescribed component-wise nodal numbers an infinite sequence of radial nodal solutions was constructed with the whole sequence of solutions sharing the same nodal data.
Another goal of the current paper is to study multiple non-radial nodal solutions.

Now let us describe the main results of the paper for system \eqref{eq:main}.
Here, $\Omega$ is a radially symmetric domain in $\mathbb{R}^n$ with $n=2$ or $3$, i.e., $\Omega$ is either a ball or an annulus.

Let $p$ be a prime factor of $N$ and write $N = pB$. We assume that $\lambda_j$, $\mu_j$, $\beta_{ij}$ satisfy
\begin{enumerate}[(A)]
	\item\label{itm:as_a} $\lambda_{p b-p+1}=\lambda_{p b-p+2}=\cdots=\lambda_{p b}>0$ for $b=1, \ldots, B$.
	\item For $i, j=1, \ldots, N$ and $i \neq j, \beta_{i j}=\beta_{j i} \leq 0$ and $\mu_j>0$.
	\item For $b=1 \ldots, B, \mathcal{B}=\left(\beta_{i j}\right)_{N \times N}$ is invariant under the action of
		\[
			\prod_{i=1}^{p-1} C_{p b-p+i, p b-p+i+1} \circ R_{p b-p+i, p b-p+i+1},
		\]
	where $R_{ij}$ is the transformation of exchanging the $i$-th row and the $j$-th row of a matrix, and $C_{ij}$ is the counterpart for column exchanging.
	\item\label{itm:as_d} For $b=1, \ldots, B$ and $p b-p+1 \leq j \leq p b$, it holds
		\[
			\mu_j+\sum_{p b-p+1 \leq i \leq p b ; i \neq j} \beta_{i j} \leq 0 .
		\]
\end{enumerate}

We define
\begin{equation*}
	\mathscr{R}_\theta u(x) = u(R_{-\theta} x),
\end{equation*}
where
\begin{equation}\label{eq:R_theta}
	R_\theta = \begin{cases}
		\begin{pmatrix}
			\cos\theta & -\sin\theta\\
			\sin\theta & \cos\theta
		\end{pmatrix} & \text{if $n=2$,}\\
		\begin{pmatrix}
			\cos\theta & -\sin\theta & 0\\
			\sin\theta & \cos\theta & 0\\
			0 & 0 & 1
		\end{pmatrix} & \text{if $n=3$.}
	\end{cases}
\end{equation}
For $U = (u_1,\dots,u_N)\in (H^1_0(\Omega))^N$, set $\mathscr{R}_\theta U=(\mathscr{R}_\theta u_1,\dots, \mathscr{R}_\theta u_N)$.
\begin{theorem}\label{th:main}
    Fix an integer $k\geq 1$.
    Then under the assumptions \ref{itm:as_a}-\ref{itm:as_d}, Problem~\eqref{eq:main} admits an unbounded sequence of positive solutions $\{(u_{1,l},\cdots,u_{N,l})\colon l\in \mathbb{N}\}$ such that each $u_{i,l}$ is $\mathscr{R}_{2\pi/k}$-invariant but not $\mathscr{R}_{2\pi/(pk)}$-invariant. In particular, $u_{i,l}$ is non-radial.
\end{theorem}
By taking $k = p^j$, we have the following corollary.
\begin{corollary}\label{coro:positive}
    Fix an integer $j\geq 0$.
    Then under the assumptions \ref{itm:as_a}-\ref{itm:as_d}, Problem~\eqref{eq:main} admits an unbounded sequence $S^{(j)} = \{(u_{1,l}^{(j)},\cdots,u_{N,l}^{(j)})\colon l\in\mathbb{N}\}$ of positive solutions such that each $u_{i,l}^{(j)}$ is $\mathscr{R}_{2\pi/p^j}$-invariant but not $\mathscr{R}_{2\pi/p^{j+1}}$-invariant.
	In particular, $u_{i,l}^{(j)}$ is non-radial.
Furthermore, for all $j_1\neq j_2$, we have $S^{(j_1)} \cap S^{(j_2)} = \varnothing$.
\end{corollary}

\begin{remark}
	We remark that in \cite{weiNonradialSymmetricBound2007} for two equations it was proved that in the entire space setting there is a ground state solution in each symmetric subspace corresponding to the integer $k$ which is a non-radial positive solution of \eqref{eq:main}, while our result above shows that within each symmetric subspace corresponding to the integer $k$ there exists an infinite sequence of non-radial positive solutions. Furthermore, with $k= p^j$ for different $j\geq 0$, the infinite sequences $S^{(j_1)}$ and $S^{(j_2)}$ are mutually different when $j_1\neq j_2$.
	We do not know whether our multiplicity results would hold for the setting of the entire space.
\end{remark}

Our methods can be adapted to the studies of nodal (sign-changing) solutions.
As an analogue of Theorem~\ref{th:main}, we have the following results for nodal solutions.

\begin{theorem}\label{th:main_nod}
	Fix an integer $k\geq 1$. Then under the assumptions \ref{itm:as_a}-\ref{itm:as_d}, Problem~\eqref{eq:main} admits an unbounded sequence of nodal solutions $\{(u_{1,l},\cdots,u_{N,l})\colon l\in \mathbb{N}\}$ such that each $u_{i,l}$ is $\mathscr{R}_{2\pi/k}$-invariant but not $\mathscr{R}_{2\pi/(2k)}$-invariant. In particular, $u_{i,l}$ is non-radial.
\end{theorem}
By taking $k = 2^j$, we have the following corollary.
\begin{corollary}\label{coro:nod}
    Fix an integer $j\geq 0$. Then under the assumptions \ref{itm:as_a}-\ref{itm:as_d}, Problem~\eqref{eq:main} admits an unbounded sequence $S^{(j)} = \{(u_{1,l}^{(j)},\cdots,u_{N,l}^{(j)})\colon l\in\mathbb{N}\}$ of nodal solutions such that each $u_{i,l}^{(j)}$ is $\mathscr{R}_{2\pi/2^j}$-invariant but not $\mathscr{R}_{2\pi/2^{j+1}}$-invariant.
	In particular, $u_{i,l}^{(j)}$ is non-radial. Furthermore, for all $j_1\neq j_2$, we have $S^{(j_1)} \cap S^{(j_2)} = \varnothing$.
\end{corollary}
For $u\in H^1_0(\Omega)$ and $\omega\in (0,+\infty)$, we say that $u$ is periodic with period $\omega$ if $\mathscr{R}_{\omega} u = u$ a.e.\ in $\Omega$. If $\omega$ is the minimal positive number with this property, it is said to be the \emph{minimal period} of $u$.

We have proved that \eqref{eq:main} has infinitely many solutions with period $2\pi/k$ for each positive integer $k$.
It is of interest to know whether \eqref{eq:main} admits solutions with minimal period $2\pi/k$.
For the $2$-coupled system
\begin{equation}\label{eq:2coupled_l}
	\begin{cases}
		-\Delta u + \lambda u = u^3 + \beta uv^2 &\text{in $\Omega$,}\\
		-\Delta v + \lambda v = v^3 + \beta vu^2 &\text{in $\Omega$,}\\
		u=v=0 &\text{on $\partial\Omega$,}\\
	\end{cases}
\end{equation}
we have the following result.
\begin{theorem}\label{th:period}
	Assume $\lambda>0$ and $\beta<0$. For each positive integer $k$, the problem~\eqref{eq:2coupled_l} admits a positive solution $(u,v)$ such that $u,v$ has minimal period $2\pi/k$.
\end{theorem}

The paper is organized as follows.
In Section~\ref{sec:multiple_positive_solutions}, we prove the results on non-radial positive solutions. Section~\ref{sec:multiple_nodal_solutions} is devoted to the proofs on non-radial nodal solutions.
In Section~\ref{sec:min_period}, we study the minimal period of the solutions in the angular variable.
Finally, in Section~\ref{sec:ext}, we discuss some further extensions of our main results.

\section{Multiple Non-radial Positive Solutions}
\label{sec:multiple_positive_solutions}
\subsection{The Variational Framework}
For $u\in H^1_0(\Omega)$, we set $\norm{u}_{i}^2 = \int_{\Omega}(\abs{\nabla{u}}^2+\lambda_i u^2)$.
The functional corresponding to~\eqref{eq:main} is
\[
	E^+(u_1,\dots,u_N) = \sum_{i = 1}^{N} \left(\frac{1}{2}\norm{u_i}_i^2-\frac{\mu_i}{4} \int_{\Omega}(u_i^+)^4\right) - \frac{1}{2}\int_{\Omega}\sum_{i\neq j}^N \beta_{ij} u_i^2u_j^2
\]
where $u^+ = \max(u,0)$. It is known (see for instance \cite{tianMultipleSolitaryWave2011}*{Lemma 2.1}) that $E^+$ is a $C^2$ functional and that every nontrivial critical point of $E^+$ is a positive classical solution of~\eqref{eq:main}.

Set for each positive integer $k$ the subspace
\begin{align*}
	\mathcal{M}_k^+ = \{U=(u_1,\dots, u_N)\in(H^1_0(\Omega))^N \colon
	&\text{$u_i = \mathscr{R}_{2\pi/k} u_i$ for $i=1,2,\dots,N$} \\
	&\text{and $u_{j+1} = \mathscr{R}_{2\pi/(pk)}u_{j}$ for $p\nmid j$}\}.
\end{align*}
For example, when $p=2$ and $B=2$, the element $(u_1,u_2,u_3,u_4)\in\mathcal{M}_k^+$ is of the form
\[
	(u_1, \mathscr{R}_{2\pi/(2k)}u_1, u_3, \mathscr{R}_{2\pi/(2k)}u_3)
\]
and each $u_j$ is $\mathscr{R}_{2\pi/k}$-invariant.

Next we define the Nehari-type manifold in $\mathcal{M}_k^+$ by
\begin{gather*}
	\mathcal{N}_k^+ = \{U=(u_1,\dots, u_N)\in \mathcal{M}_k^+ \colon u_j\neq 0,\partial_j E^+ (U)u_j=0,\forall j=1,2,\dots,N\},
\end{gather*}
where
\[
	\partial_j E^+(U)u_j = \int_{\Omega} (\abs{\nabla u_j}^2 + \lambda_j u_j^2) - \mu_j\int_{\Omega}(u_j^+)^4 - \int_{\Omega}\sum_{\substack{k=1\\k\neq j}}^N \beta_{jk}u_j^2 u_k^2.
\]
Under the assumptions~\ref{itm:as_a}-\ref{itm:as_d}, the functional $E^+$ possesses a $\mathbb{Z}_p = \langle\sigma\mid\sigma^p=\id\rangle$ symmetry in the sense that
\[
	E^+(\sigma u) = E^+(u)
\]
for each $u\in (H^1_0(\Omega))^N$.
Here and throughout the paper, $\sigma$ denotes the permutation given by
\begin{equation}\label{eq:action_of_zp}
	\begin{aligned}
		&\sigma (u_1,u_2,\dots,u_p;\dots;u_{N-p+1},u_{N-p+2},\dots,u_N)\\
		=&(u_2,u_3,\dots,u_p,u_1;\dots;u_{N-p+2},u_{N-p+3},\dots,u_N,u_{N-p+1}).
	\end{aligned}
\end{equation}
\begin{lemma}
	The subspace $\mathcal{M}_k^+$ and the submanifold $\mathcal{N}_k^+$ are natural constraints, i.e., every constrained critical point of $E^+$ on them is also a critical point of $E^+$.
\end{lemma}
\begin{proof}
	It is easy to see that $\mathcal{M}_k^+$ is a fixed point space of an isometric representation of some group. Indeed, we can define the action of $\mathbb{Z}_{pk} = \langle g \mid g^{pk} = \id \rangle$ on $(H^1_0(\Omega))^N$ as
	\[
		g \circ (u_1,\dots,u_N) = \mathscr{R}_{2\pi/{pk}}\sigma (u_1,\dots,u_{N}).
	\]
	Note that $\sigma^p = \id$ and hence that $g^p \circ (u_1, \dots, u_N) = \mathscr{R}_{2\pi/{k}}(u_1, \dots, u_N)$, so the fixed point space of the action is precisely $\mathcal{M}_k^+$. Then it follows from the symmetric criticality principle that $\mathcal{M}_k^+$ is a natural constraint. The rest of the proof is similar to that in \cite{dancerPrioriBoundsMultiple2010} and \cite{tianMultipleSolitaryWave2011}.
\end{proof}
This lemma makes it legitimate to reduce the problem to seeking critical points on $\mathcal{N}_k^+$.

The Palais-Smale condition also holds for $E^+$.
\begin{lemma}\label{lm:ps}
	The restricted functional $E^+\mid_{\mathcal{N}_k^+}$ satisfies the Palais-Smale condition, i.e., each sequence $(u_j)\subset \mathcal{N}_k^+$ such that $E(u_j)$ is bounded and $\nabla E^+\mid_{\mathcal{N}_k^+}\to 0$ has a convergent subsequence.
\end{lemma}
The proof of this lemma is exactly same as that in \cite{tianMultipleSolitaryWave2011}.
\subsection{A $\mathbb{Z}_p$ Index}
In this section, we introduce a $\mathbb{Z}_p$ index theory, which is used in the estimate of the number of critical points. The classical works here are \cites{wangIndexTheory1990,wangBorsukUlamTheorem1989}.

Now we define an index associated with $\mathbb{Z}_p$, where the action of $\mathbb{Z}_p = \langle\sigma\mid\sigma^p=\id\rangle$ is defined by $\sigma$ as~\eqref{eq:action_of_zp}.
\begin{definition}
	For any closed $\sigma$-invariant subset $A\subset\mathcal{N}_k^+$, define the index
	\[
		\gamma(A) = \min \Big(\{\infty\}\cup\{m\in\mathbb{N}\colon \text{$\exists h\in C(A;\mathbb{C}^m\setminus\{0\})$ satisfying $h(\sigma U) = \mathrm{e}^{2\pi\mathrm{i}/p} h(U)$}\}\Big).
	\]
\end{definition}
In particular, $\gamma(\varnothing) = 0$ and $\gamma(A) = \infty$ if $A$ contains a fixed point of $\sigma$.

In order to use the $\mathbb{Z}_p$ index theory, we need the following lemma to exclude the existence of fixed point under the action of $\mathbb{Z}_p$.
\begin{lemma}\label{lm:no_fixed_point}
	Under the assumptions \ref{itm:as_a}-\ref{itm:as_d}, for $p\mid j$, $u_{j+1}=u_{j+2}=\dots=u_{j+p}$ cannot hold for $(u_1,\dots,u_N)\in^\mathcal{N}_k+$. Therefore, there is no fixed point on $\mathcal{N}_k^+$ under the action of $\mathbb{Z}_p$.
\end{lemma}
\begin{proof}
	Suppose the assertion of the lemma is false. Without loss of generality, assume that $U=(u_1,\dots,u_N)\in\mathcal{N}_k^+$ and $u_1 = u_2 = \dots = u_p$. By the assumption~\ref{itm:as_d} and the definition of $\mathcal{N}_k^+$, we obtain
	\begin{align*}
		0<{}& \int_{\Omega} (\abs{\nabla u_1}^2 + \lambda_1 u_1^2)\\
		={}& \mu_1 \int_{\Omega}(u_1^+)^4 + \int_{\Omega}\sum_{\substack{k=2}}^N \beta_{jk}u_1^2 u_k^2\\
		\leq{}& \left(\mu_1 + \sum_{k=2}^N \beta_{ik}\right) \int_{\Omega}u_1^4\\
		\leq{}&0,
	\end{align*}
	which is impossible.
\end{proof}
\begin{remark}\label{rem:sym}
	As $u_{j+1} = \mathscr{R}_{2\pi/(pk)}u_{j}$ for $p\nmid j$ in $\mathcal{N}_k^+$, this lemma implies that functions in $\mathcal{N}_k^+$ cannot be $\mathscr{R}_{2\pi/(pk)}$-invariant.
\end{remark}
Let $\mathcal{N}_k^c = \{U\in\mathcal{N}_k^+\colon E^+(U)\leq c\}$. The Palais-Smale condition ensures the validity of the following deformation lemma.
\begin{proposition}\label{prop:index1}
	Let $c\in\mathbb{R}$, and let $\mathcal{O}$ be a $\sigma$-invariant open neighborhood of $K_c$ in $\mathcal{N}_k^+$. Then there exists $\varepsilon>0$ and a $C^1$-deformation $\eta\colon [0,1]\times \mathcal{N}_k^{c+\varepsilon}\setminus\mathcal{O}\to\mathcal{N}_k^{c+\varepsilon}$ such that
	\begin{enumerate}[(a)]
		\item $\eta(0,\cdot) = \id$,
		\item $\eta(1, \mathcal{N}_k^{c+\varepsilon}\setminus\mathcal{O}) \subset \mathcal{N}_k^{c-\varepsilon}$,
		\item $\eta(t,\sigma(\cdot)) = \sigma\eta(t,\cdot)$ for each $t\in[0,1]$.
	\end{enumerate}
\end{proposition}
With the deformation lemma, one can prove the following elementary properties of the index.
\begin{proposition}\label{prop:properties_of_index}
    Let $A, B\subset\mathcal{N}_k^+$ be closed and $\sigma$-invariant.
    \begin{enumerate}[(a)]
        \item If $A\subset B$, then $\gamma(A)\leq\gamma(B)$.
        \item $\gamma(A\cup B)\leq \gamma(A) + \gamma(B)$.
        \item If $g\colon A\to\mathcal{N}_k^+$ is continuous and $\sigma$-equivariant, i.e.,
			\[
				g(\sigma(U)) = \sigma g(U), \qquad \text{for all $U\in A$},
			\]
			then $\gamma(A)\leq\gamma(\overline{g(A)})$.
        \item If $\gamma(A)>1$ and $A$ does not contain fixed points of $\sigma$, then $A$ is an infinite set. \label{it:index_d}
        \item If $A$ is compact and $A$ does not contain fixed points of $\sigma$, then $\gamma(A)<\infty$, and there exists a relatively open and $\sigma$-invariant neighborhood $N$ of $A$ in $\mathcal{N}_k^+$ such that $\gamma(A) = \gamma(\bar{N})$.
        \item If $S$ is the boundary of a bounded and $\sigma$-invariant neighborhood of zero in an $m$-dimensional complex normed vector space and $\Psi\colon S\to\mathcal{M}$ is a continuous map satisfying $\Psi(e^{2\pi i/p} U)=\sigma(\Psi(U))$, then $\gamma(\Psi(S))\geq m$. \label{it:index}
    \end{enumerate}
\end{proposition}

Define
\begin{equation*}
	c_j = \inf \{c\in\mathbb{R}\colon \gamma(\mathcal{N}_k^c)\geq j\}
\end{equation*}
for each positive integer $j$.
\begin{proposition}
	For all $c\in\mathbb{R}$, we have $\gamma(K_c)<\infty$, and there exists $\varepsilon>0$ such that
	\[
		\gamma(\mathcal{N}_k^{c+\varepsilon}) \leq \gamma(\mathcal{N}_k^{c-\varepsilon}) + \gamma(K_c).
	\]
\end{proposition}
\begin{proposition}\label{prop:index_of_critical_points}
	If $c:=c_j=c_{j+1}=\dots=c_{j+d}<\infty$, then $\gamma(K_c)\geq d+1$.
\end{proposition}
The proofs of these propositions are similar to those in \cite{tianMultipleSolitaryWave2011} and \cite{dancerPrioriBoundsMultiple2010}, and we omit them.
\subsection{Multiplicity Result}
We can derive the existence of infinitely many solutions by showing each $c_j<\infty$.
\begin{proposition}\label{prop:spl_pos}
	Denote by $\mathbb{S}^{2m-1}$ the unit sphere in $\mathbb{C}^m$. For each positive integer $m$, there exists a continuous map $\psi\colon \mathbb{S}^{2m-1} \to\mathcal{N}_k^+$, such that
	\[
		\psi( e^{2\pi i/p} z) = \sigma\psi(z).
	\]
\end{proposition}
\begin{proof}
	Let
	\[
		\Omega_\mathrm{sec} = \begin{cases}
			\{(r\cos\theta,r\sin\theta)\in \Omega\colon 0<\theta< \frac{2\pi}{pk}\} & \text{if $n=2$,}\\
			\{(r\cos\theta,r\sin\theta,z)\in \Omega\colon 0<\theta< \frac{2\pi}{pk}\} & \text{if $n=3$.}
		\end{cases}
	\]
	For each $i \in \{1,\dots, m\}$, $j\in\{1, \dots, B\}$, choose a radially symmetric domain $\Omega_{i,j}\subset\Omega$ and a function $\hat U_i^{(j)}\in C_0^\infty (\Omega_\mathrm{sec}\cap\Omega_{i,j})\setminus\{0\}$ such that all the $\Omega_{i,j}$-s are pairwise disjoint.

	Next we put
	\[
		U_i^{(j)} = \sum_{s = 1}^k \mathscr{R}_{2\pi s/k} \hat U_i^{(j)} \qquad \text{for $i=1,\dots, m$, $j=1,\dots,B$.}
	\]
	Thus, we have
	\begin{enumerate}[(U1)]
		\item\label{enu:u1} $U_i^{(j)}$ is $\mathscr{R}_{2\pi/k}$-invariant,
		\item\label{enu:u2} $\supp \mathscr{R}_{\theta_1} U_{i_1}^{(j_1)}\cap \supp \mathscr{R}_{\theta_2} U_{i_2}^{(j_2)} = \varnothing$ for all $\theta_1,\theta_2\in\mathbb{R}$ and $(i_1,j_1)\neq(i_2,j_2)$.
		\item\label{enu:u3} $\supp \mathscr{R}_{\frac{2\pi s}{pk}} U_{i}^{(j)} \cap \supp \mathscr{R}_{\frac{2\pi t}{pk}} U_{i}^{(j)} = \varnothing$ if $s\neq t$ and $s,t\in \{0,1,\dots,p-1\}$,
	\end{enumerate}

	Define $\psi_0\colon \mathbb{S}^{2m-1}\to (H_0^1(\Omega))^N$
	by
	\begin{align*}
		&{} \psi_0(z)
		= \psi_0(r_1\mathrm{e}^{\mathrm{i}\theta_1}, \dots, r_m\mathrm{e}^{\mathrm{i}\theta_m})\\
		=&{} \left(
			U^{(1)}, \mathscr{R}_{\frac{2\pi}{pk}}U^{(1)},\dots,\mathscr{R}_{\frac{2\pi(p-1)}{pk}}U^{(1)};\quad
			\dots;\quad
			U^{(B)}, \mathscr{R}_{\frac{2\pi}{pk}}U^{(B)},\dots,\mathscr{R}_{\frac{2\pi(p-1)}{pk}}U^{(B)}
		\right)
	\end{align*}
	where
	\[
		U^{(j)} = \sum_{i = 1}^m  r_i \mathscr{R}_{\theta_i/k} U_i^{(j)}, \qquad \text{for $j=1,\dots, B$.}
	\]
	Since $U_i^{(j)}$ are $\mathscr{R}_{2\pi/k}$-invariant, the functions $U^{(j)}$ are well-defined.
	Then we have
	\begin{align*}
		\psi_0(\mathrm{e}^{2\pi\mathrm{i}/p} z) &= \Phi(r_1\mathrm{e}^{\mathrm{i}(\theta_1+2\pi\mathrm{i}/p)}, \dots, r_m\mathrm{e}^{\mathrm{i}(\theta_m+2\pi\mathrm{i}/p)})\\
		&= \mathscr{R}_{2\pi/(pk)}\psi_0(z) \\
		&= \sigma\psi_0(z).
	\end{align*}

	From \ref{enu:u1}-\ref{enu:u3}, it follows that all components of $\psi_0(z)$ are $\mathscr{R}_{2\pi/k}$-invariant and their supports are pairwise disjoint,
	and hence that $\psi_0(\mathbb{S}^{2m-1})\subset \mathcal{M}_k^+\setminus\{0\}$.

	Next we construct a $C^1$ map $\Lambda \colon \psi_0(\mathbb{S}^{2m-1})\to\mathcal{N}_k^+$ such that $\Lambda\sigma = \sigma\Lambda$.
	For each $(U_1,\dots,U_N)\in\psi_0(\mathbb{S}^{2m-1})$,
	note that the supports of each component are pairwise disjoint, so by a direct computation we have
	\[
		\Lambda(U_1,\dots, U_N) := \left(\frac{\norm{U_1}_{1}}{\norm{U_1^+}^2_{L^4(\Omega)}}U_1, \dots, \frac{\norm{U_N}_{N}}{\norm{U_N^+}^2_{L^4(\Omega)}}U_N\right)\in \mathcal{N}_k^+.
	\]
	
	Since the diagram
	\[
		\begin{tikzcd}
			\mathbb{S}^{2m-1} \arrow[d, "e^{2\pi i/3}\cdot\id"] \arrow[r, "\psi_0"] & \psi_0(\mathbb{S}^{2m-1}) \arrow[r, "\Lambda"] \arrow[d, "\sigma"] & \mathcal{N}_k^+ \arrow[d, "\sigma"] \\
			\mathbb{S}^{2m-1} \arrow[r, "\psi_0"]                                   & \psi_0(\mathbb{S}^{2m-1}) \arrow[r, "\Lambda"]                     & \mathcal{N}_k^+
		\end{tikzcd}
	\]
	commutes, the proof is completed by setting $\psi = \Lambda\circ\psi_0$.
\end{proof}
Together with Lemma~\ref{lm:no_fixed_point} and \ref{it:index} of Proposition \ref{prop:properties_of_index}, this proposition ensures the existence of the set with an arbitrarily large index, and hence we have $c_j<\infty$ for each positive integer $j$. Now we can complete the proof of Theorem~\ref{th:main}.
\begin{proof}[Proof of Theorem~\ref{th:main}]
	Recall that every nontrivial critical point of $E^+$ is a positive classical solution of~\eqref{eq:main}. Combining Proposition~\ref{prop:properties_of_index} and Proposition~\ref{prop:index_of_critical_points} gives the existence of infinitely many positive solutions in $\mathcal{N}_k^+$ for each positive integer $k$. By Remark~\ref{rem:sym} and the definition of $\mathcal{N}_k^+$, we conclude that these solutions $\mathscr{R}_{2\pi/k}$-invariant but not $\mathscr{R}_{2\pi/(pk)}$-invariant, and the proof is completed.
\end{proof}
By taking $k = p^j$ for $j=0,1,2,...$, we obtain Corollary~\ref{coro:positive}.
\section{Multiple Non-radial Nodal Solutions}
\label{sec:multiple_nodal_solutions}
In this section, we set
\begin{gather*}
	E(u_1,\dots,u_N) = \sum_{i = 1}^{N} \left(\frac{1}{2}\norm{u_i}_i^2-\frac{\mu_i}{4} \int_{\Omega}(u_i)^4\right) - \frac{1}{2}\int_{\Omega}\sum_{i\neq j}^N \beta_{ij} u_i^2u_j^2,\\
	\begin{aligned}
		\mathcal{M}_k^{\mathrm{nod}} = \{U=(u_1,\dots, u_N)\in(H^1_0(\Omega))^N \colon
		&\text{$u_j = -\mathscr{R}_{2\pi/2k} u_j$}\\
		&\text{for $j=1,2,\dots,N$} \},
	\end{aligned}
\end{gather*}
and the Nehari-type manifold
\[
	\mathcal{N}_k^{\mathrm{nod}} = \{U=(u_1,\dots, u_N)\in \mathcal{M}_k^{\mathrm{nod}} \colon u_j\neq 0,\partial_j E(U)u_j=0,\forall j=1,2,\dots,N\},
\]
where
\[
	\partial_j E(U)u_j = \int_{\Omega} (\abs{\nabla u_j}^2 + \lambda_j u_j^2) - \mu_j\int_{\Omega}(u_j)^4 - \int_{\Omega}\sum_{\substack{k=1\\k\neq j}}^N \beta_{jk}u_j^2 u_k^2, \qquad \forall j=1,2,\dots,N.
\]
The functional $E$ is of $C^2$ and every nontrivial critical point of $E$ is a classical solution of~\eqref{eq:main}.

\begin{remark}
	According to the definition of $\mathcal{M}_k^{\mathrm{nod}}$, every component of $u\in \mathcal{M}_k^{\mathrm{nod}}$ is $\mathscr{R}_{2\pi/k}$-invariant, since $u_j = -\mathscr{R}_{2\pi/2k} u_j$ implies $u_j = \mathscr{R}_{2\pi/k} u_j$.
\end{remark}
Similarly to that in the previous section, since $\mathcal{M}_k^{\mathrm{nod}}$ is a fixed point space of an isometric representation of group $\mathbb{Z}_{2k} = \langle g \mid g^{2k} = \id \rangle$ defined as
\[
	g \circ (u_1,\dots,u_N) = -\mathscr{R}_{2\pi/{2k}}(u_1,\dots,u_{N}),
\]
we have the following lemma.
\begin{lemma}
	The subspace $\mathcal{M}_k^{\mathrm{nod}}$ and the submanifold $\mathcal{N}_k^{\mathrm{nod}}$ are natural constraints, i.e., every constrained critical point of $E$ on them is also a critical point of $E$.
\end{lemma}
To use the $\mathbb{Z}_p$ index theory, we need the counterparts of Proposition~\ref{prop:index1}-\ref{prop:spl_pos}. These counterparts can be derived using the same method, except for Proposition~\ref{prop:spl_pos}, where the situation is more delicate. We present the details separately as follows.
\begin{proposition}
	Denote by $\mathbb{S}^{2m-1}$ the unit sphere in $\mathbb{C}^m$. For any positive integer $m$, there exists a continuous map $\psi\colon \mathbb{S}^{2m-1} \to\mathcal{N}_k^{\mathrm{nod}}$, such that
	\[
		\psi( e^{2\pi i/p} z) = \sigma\psi(z).
	\]
\end{proposition}
\begin{proof}
	For each $i \in \{1,\dots, m\}$, $j\in\{1, \dots, B\}$ and $\theta\in \mathbb{R}/(2\pi\mathbb{Z})$, one can find $U^{i,j}_\theta\in C_0^\infty(\Omega)\setminus\{0\}$ satisfying the following conditions.
	\begin{enumerate}[(U1$'$)]
		\item\label{enu:u1_prime} The map $\mathbb{R}/(2\pi\mathbb{Z})\to C_0^\infty(\Omega)\colon \theta\mapsto U_\theta^{i,j}$ is continuous and $U_\theta^{i,j} = -\mathscr{R}_{2\pi/2k} U_\theta^{i,j}$.
		\item\label{enu:u2_prime} $\supp U^{i_1,j_1}_{\theta_1}\cap \supp U^{i_2,j_2}_{\theta_2} = \varnothing$ for all $\theta_1,\theta_2\in\mathbb{R}/(2\pi\mathbb{Z})$ and $(i_1,j_1)\neq(i_2,j_2)$.
		\item\label{enu:u3_prime} $\supp U^{i,j}_{\theta} \cap \supp U^{i,j}_{\theta+\frac{2\pi s}{p}} = \varnothing$ if $\theta\in\mathbb{R}/2\pi \mathbb{Z}$ and $s \in \{1,\dots,p-1\}$.
	\end{enumerate}
	To be more precise, for each $i \in \{1,\dots, m\}$, $j\in\{1, \dots, B\}$ and $\ell\in\{1,2,\dots,2p\}$, find a nonzero function $\varphi_\ell^{i,j}\in C_0^\infty(\Omega)$ such that $\varphi_\ell^{i,j} = -\mathscr{R}_{2\pi/2k} \varphi_\ell^{i,j}$ and all the supports of $\varphi_\ell^{i,j}$ are pairwise disjoint.
	Then choose $2p$ functions $\eta_1,\dots,\eta_{2p}\in C_0^\infty(\mathbb{R}/(2\pi\mathbb{Z}))$ such that for each $\ell\in\{1,2,\dots,2p\}$
	\[
		\eta_\ell>0 \text{~in} \left(\frac{(\ell-1)\pi}{p}, \frac{(\ell+1)\pi}{p}\right)+2\pi\mathbb{Z} \subset \mathbb{R}/(2\pi\mathbb{Z}),
	\]
	and vanishes outside.
	Then we set
	\[
		U_\theta^{i,j} = \sum_{\ell=1}^{2p} \eta_\ell(\theta)\varphi_\ell^{i,j}.
	\]
	Thus, $U_\theta^{i,j}$ satisfies \ref{enu:u1_prime}-\ref{enu:u3_prime}. The validity of \ref{enu:u1_prime} and \ref{enu:u2_prime} are obvious, while \ref{enu:u3_prime} follows from our choice of $\supp\eta_\ell$.
	In fact, when $\theta\in (\frac{t\pi}{p}, \frac{(t+1)\pi}{p})$ for some $t\in\{0,1,\dots,2p-1\}$, $\eta_\ell\neq 0$ if and only if $\ell \equiv t$ or $t+1 \pmod{2p}$. With setting $\varphi^{i,j}_{\ell+2p} = \varphi^{i,j}_{\ell}$ for $\ell=1,2,\dots,2p$, we conclude that $U_\theta^{i,j}\neq 0$ and that
	\begin{equation*}
		\supp U^{i,j}_{\theta} \subset \left(\supp\varphi^{i,j}_t \cup \supp\varphi^{i,j}_{t+1}\right).
	\end{equation*}
	Similarly, for $s\in\{1,2,\dots,p-1\}$, we conclude that
	\[
		\supp U^{i,j}_{\theta+\frac{2\pi s}{p}} \subset \left(\supp\varphi^{i,j}_{t+2s} \cup \supp\varphi^{i,j}_{t+2s+1}\right),
	\]
	and hence that
	\[
		\supp U^{i,j}_{\theta} \cap \supp U^{i,j}_{\theta+\frac{2\pi s}{p}} = \varnothing.
	\]
	The case of $\theta = \frac{t\pi}{p}$ for some $t$ is similar, and we omit it.

	Hence, we can define $\psi_0\colon \mathbb{S}^{2m-1}\to (H_0^1(\Omega))^N$
	by
	\begin{align*}
		&{} \psi_0(z)
		= \psi_0(r_1\mathrm{e}^{\mathrm{i}\theta_1}, \dots, r_m\mathrm{e}^{\mathrm{i}\theta_m})\\
		=&{} \left(
			U^{(1)}_1, \dots, U^{(1)}_p;\quad
			\dots;\quad
			U^{(B)}_1, \dots, U^{(B)}_p
		\right)
	\end{align*}
	where
	\[
		U^{(j)}_\ell = \sum_{i = 1}^{m} r_i U_{\theta_i + \frac{2\pi \ell}{p}}^{i,j} , \qquad \text{for $\ell=1,\dots,p$, $j=1,\dots, B$.}
	\]
	We thus have
	\begin{align*}
		\psi_0(\mathrm{e}^{2\pi\mathrm{i}/p} z) &= \Phi(r_1\mathrm{e}^{\mathrm{i}(\theta_1+2\pi\mathrm{i}/p)}, \dots, r_m\mathrm{e}^{\mathrm{i}(\theta_m+2\pi\mathrm{i}/p)})\\
		&= \sigma\psi_0(z).
	\end{align*}

	From \ref{enu:u1_prime}-\ref{enu:u3_prime}, it follows that $\psi_0(z)\in\mathcal{M}_k^{\mathrm{nod}}$, and that the supports of all components of $\psi_0(z)$ are pairwise disjoint.
	Since $r_i$-s are not all zero for $(r_1\mathrm{e}^{\mathrm{i}\theta_1}, \dots, r_m\mathrm{e}^{\mathrm{i}\theta_m}) \in \mathbb{S}^{2m-1}$, each component of $\psi_0(z)$ is nonzero by the definition. These facts allow one to construct a $C^1$ map $\Lambda \colon \psi_0(\mathbb{S}^{2m-1})\to\mathcal{N}_k$ such that $\Lambda\sigma = \sigma\Lambda$.
	Then the rest part is similar to that in Proposition~\ref{prop:spl_pos}.
\end{proof}

Now we can finish the proof of Theorem~\ref{th:main_nod}.
\begin{proof}[Proof of Theorem~\ref{th:main_nod}]
	The existence part of the theorem can be obtained using the same method. By the definition of $\mathcal{M}_k^{\mathrm{nod}}$, $u\in\mathcal{M}_k^{\mathrm{nod}}$ cannot be $\mathscr{R}_{2\pi/(2k)}$-invariant provided $u\neq 0$, and this gives the last part of the theorem.
\end{proof}
By taking $k = 2^j$ for $j=0,1,2,...$, we obtain Corollary~\ref{coro:nod}.
\section{Minimal Period}
\label{sec:min_period}
In this section, we will restrict ourselves to $2$-coupled system~\eqref{eq:2coupled_l}.
Without loss of generality, we assume $\lambda = 1$. Then the system becomes
\begin{equation}\label{eq:2coupled}
	\begin{cases}
		-\Delta u + u = u^3 + \beta uv^2\\
		-\Delta v + v = v^3 + \beta vu^2.
	\end{cases}
\end{equation}
We follow the notations used in Section~\ref{sec:multiple_positive_solutions}.
We will show that, for $\beta < 0$, each component of the least energy solution of~\eqref{eq:2coupled} on $\mathcal{N}_k$ has a minimal period of $2\pi/k$.
\subsection{The pull-back on $H^1_0(\Omega)$ of the Functional}
Our proof starts with the following observation.
\begin{lemma}
	The subspace $\mathcal{M}_k^+$ is toplinear isomorphic to $H_0^1(\Omega)$ in the sense that there is a continuous bijective linear map from $H_0^1(\Omega)$ onto $\mathcal{M}_k^+$.
\end{lemma}
\begin{proof}
	Using polar coordinate system, we define $\Psi_k\colon H_0^1(\Omega) \to \mathcal{M}_k^+$ by
	\begin{align*}
		(\Psi_k u)(r,\theta) = (u(r,k\theta), u(r,k\theta + \pi/k)).
	\end{align*}
	It is easy to check that the map $\Psi_k$ is well-defined and has the desired properties.
\end{proof}

From now on, $\Psi_k$ denotes the map given above. Let $\hat E_k^+$, the pullback of $E^+$ by $\Psi_k$, be defined by $\hat E_k = E\circ\Psi_k$. A direct computation gives
\begin{align*}
	\hat E_k^+(u) =& \int_0^\infty\int_0^{2\pi} \left(u_r^2+\frac{k^2}{r^2}u_{\theta}^2\right)r \,\dd\theta\,\dd r
	+ \int_{\Omega} u^2 - \frac{1}{2}\int_{\Omega}(\abs{u^+}^4 + \beta u^2(x) u^2(-x))\,\dd x, 
\end{align*}
Besides, it is easily seen that, for every critical point $U$ of $E^+$, the function $\hat{u}:=\Psi_k^{-1}(U)$ is a critical point of $\hat E_k$, and that $\hat{u}$ satisfies
\begin{equation}\label{eq:reduced_eq}
	-\Delta_k \hat{u}(x) + \hat{u}(x) = \hat{u}^3(x) + \beta \hat{u}(x)\hat{u}^2(-x),
\end{equation}
where
\begin{equation}\label{eq:laplacian_k}
	\Delta_k u(r,\theta) := \frac{1}{r}\frac{\partial}{\partial r}\left(r\frac{\partial u}{\partial r}\right) + \frac{k^2}{r^2}\frac{\partial^2 u}{\partial\theta^2}, \qquad \forall u\in C^2(\Omega),
\end{equation}
writing in polar coordinates.

Let $\hat{\mathcal{N}}_k^+ = \Psi_k^{-1}(\mathcal{N}_k^+)$ be the Nehari-type manifold associated with $\hat{\Psi}_k$. We have the following lemma.
\begin{lemma}\label{lm:nehari_rds}
	Assume $\beta<0$. Let $u\in H^1_0(\Omega)$ satisfy
	\begin{equation}\label{eq:asp_nehari_rds}
		\int_{\Omega} \abs{u^+}^4 + \beta \int_{\Omega} u^2(x) u^2(-x)\,\dd x >0.
	\end{equation}
	Then there is a unique $\lambda>0$ such that $\lambda u\in\hat{\mathcal{N}}_k^+$ and
	\[
		\hat{E}_k^+(\lambda u) = \sup_{t > 0} \hat{E}_k^+(t u).
	\]
\end{lemma}
\begin{proof}
	Inequality~\eqref{eq:asp_nehari_rds} implies
	\[
		\lim_{t\to\infty} \hat{E}_k^+(t u) = -\infty.
	\]
	On the other hand, by Sobolev inequality and Cauchy-Schwarz inequality,
	\begin{align*}
		\frac{1}{2}\int_{\Omega}(\abs{u^+}^4 + \beta u^2(x) u^2(-x))\,\dd x \leq C\left(\int_{\Omega}\abs{\nabla u}^2\,\dd x\right)^2.
	\end{align*}
	This implies $0$ is an isolated local minimum of $\hat E$. Hence, the function $t\mapsto \hat E(t u)$ achieves its maximum at some $\lambda > 0$. An easy computation shows that
	\[
		\lambda = \sqrt{\frac{\int_0^1\int_0^{2\pi} \left(u_r^2+\frac{k^2}{r^2}u_{\theta}^2\right)r \,\dd\theta\,\dd r
		+ \int_{\Omega} u^2}{\int_{\Omega} \abs{u^+}^4 + \beta \int_{\Omega} u^2(x) u^2(-x)\,\dd x}},
	\]
	which is the unique critical point of $t\mapsto \hat E(t u)$ on $(0,+\infty)$, and thus $\lambda u \in \hat{\mathcal{N}}_k$.
\end{proof}
\subsection{The Symmetry of the Minimizer}
We begin with the existence of the minimizer of $\hat{E}_k^+$ on $\hat{\mathcal{N}}_k^+$.
\begin{proposition}\label{prop:minimizer}
    Assume $\beta<0$. For each positive integer $k$, there exists $(u,v)\in \mathcal{N}_k^+$ such that $E^+(u,v) = \inf_{\mathcal{N}_k} E^+$, i.e.,
	\[
		\hat{E}_k^+(\Psi^{-1}(u,v)) = \inf_{\hat{\mathcal{N}}_k^+} \hat{E}_k^+.
	\]
\end{proposition}
\begin{proof}
	The proof is similar to that of \cite{tavaresExistenceSymmetryResults2013}*{Theorem 1.1} and will be omitted.
\end{proof}
Recall that a function $u\colon \Omega\to\mathbb{R}$ is said to be \emph{foliated Schwarz symmetry} with respect to $e\in\partial B_1(0)$ if for a.e.\ $r>0$ such that $\partial B_r(0)\subset\Omega$ and for every $c\in\mathbb{R}$ the set $\{x\in\partial B_r(0)\colon u(x)\geq c\}$ is either equal to $\partial B_r(0)$ or to a geodesic ball in $\partial B_r(0)$ centered at $rp$. In other words, $u$ is foliated Schwarz symmetry with respect to $e\in\partial B_1(0)$ if $u(x)$ depends only on $(r,\theta) = (\abs{x}, \arccos(x\cdot e)/\abs{x})$ and is non-increasing in $\theta$.

We will show the minimizer of $\hat{E}_k^+$ on $\hat{\mathcal{N}}_k^+$ is foliated Schwarz symmetric.
For this, we introduce the notions about polarization, following Tavares and Weth \cite{tavaresExistenceSymmetryResults2013}. The classical works here are \cites{vanschaftingenExplicitApproximationSymmetric2009,wethSymmetrySolutionsVariational2010}.
Define
\[
	\mathcal{H}_0 = \{H\subset\mathbb{R}^n\colon \text{~$H$ is a closed half-space in $\mathbb{R}^n$ and $0\in\partial H$}\}.
\]
For each $H\in\mathcal{H}_0$, we use $\sigma_H$ to denote the reflection with respect to $\partial H$, and for $u\in H^1(\Omega)$ we define $u_H$ by
\[
	u_H(x) = \begin{cases}
		\max\{u(x), u(\sigma_H(x))\} & x\in \Omega\cap H,\\
		\min\{u(x), u(\sigma_H(x))\} & x\in \Omega\setminus H.
	\end{cases}
\]
We can now state our main tool, which is an analogue of \cite{tavaresExistenceSymmetryResults2013}*{Theorem 4.3}.
\begin{lemma}\label{lm:fss}
	Assume $\beta<0$. Let $\hat u \in C^2(\Omega)\cap C^1(\bar\Omega)$ be a classical solution of Equation~\eqref{eq:reduced_eq}. If $\hat{u}_H$ is still a classical solution of~\eqref{eq:reduced_eq} for every $H\in \mathcal{H}_0$, then $u$ is foliated Schwarz symmetric.
\end{lemma}
One can rewrite Equation~\eqref{eq:reduced_eq} as
\[
	\begin{cases}
		-\Delta_k \hat{u} + \hat{u} = \hat{u}^3 + \beta \hat{u}\hat{v}^2,\\
		-\Delta_k \hat{v} + \hat{v} = \hat{v}^3 + \beta \hat{v}\hat{u}^2,\\
		\hat{u}(x) = \hat{v}(-x) \qquad \text{for all $x\in\Omega$}.
	\end{cases}
\]
Since the operator $\Delta_k$ given by~\eqref{eq:laplacian_k} is still a uniform elliptic operator for which the strong maximum principle holds, the same proof of~\cite{tavaresExistenceSymmetryResults2013}*{Theorem 4.3} still works.

\begin{lemma}\label{lm:polarization_ineq}
	Assume $\beta<0$. For each $u\in H^1_0(\Omega)$ and $H\in \mathcal{H}_0$, $\hat E_k^+(u_H)\leq \hat E_k^+(u)$.
\end{lemma}
\begin{proof}
	We write
	\begin{gather*}
		G(u) = 2\int_0^1\int_0^{2\pi} \left(u_r^2+\frac{k^2}{r^2}u_{\theta}^2\right)r \,\dd\theta\,\dd r,\\
		P(u)=\int_{\Omega} 2u^2-\frac{1}{2}\abs{u^+}^4\,\dd x,
		\qquad \text{and } Q(u) = -\beta \int_{\Omega} u^2(x) u^2(-x)\,\dd x.
	\end{gather*}

	Without loss of generality, we assume $H = \{(x,y)\colon x\geq 0\}$. Using polar coordinate system, we have
	\[
		u_H(r,\theta) = \begin{cases}
			\max\{u(r,\theta), u(r,-\theta)\} & \theta\in[0,\pi],\\
			\min\{u(r,\theta), u(r,-\theta)\} & \theta\in[-\pi,0].
		\end{cases}
	\]
	Thus, we have $G(u_H)=G(u)$ by a direct computation. It follows from a change of variables formula that $P(u_H)=P(u)$. It remains to show $Q(u_H)\leq Q(u)$.
	Indeed, by the rearrangement inequality and the definition of $u_H$, we have
	\begin{align*}
		&\int_{\Omega} u_H^2(x) u_{H}^2(-x)\,\dd x\\
		={}& \int_{\Omega\cap H} u_{H}^2(x) u_{H}^2(-x) + u_{H}^2(\sigma_H(x)) u_{H}^2(-\sigma_H(x))\,\dd x\\
		={}& \int_{\Omega\cap H} u_{H}^2(x) u_{H}^2(-x) + u_{H}^2(\sigma_H(x)) u_{H}^2(-\sigma_H(x))\,\dd x\\
		\leq{}& \int_{\Omega\cap H} u^2(x) u^2(-x) + u^2(\sigma_H(x)) u^2(-\sigma_H(x))\,\dd x\\
		={}& \int_{\Omega} u^2(x) u^2(-x)\,\dd x.
	\end{align*}
	Combining these gives $\hat E_k(u_H)\leq \hat E_k(u)$.
\end{proof}

\begin{lemma}\label{lm:uh_still_in_Nk}
	Assume $\beta<0$. Let $\hat{u}$ be the minimizer of $\hat{E}_k$ on $\hat{\mathcal{N}}_k$. Then $u_H$ is still on $\hat{\mathcal{N}}_k$, i.e., $\Psi_k u_H \in\mathcal{N}_k$.
\end{lemma}
\begin{proof}
	The condition $\hat{u}\in\hat{\mathcal{N}}_k$ implies that
	\[
		\int_{\Omega} \abs{\hat{u}^+}^4 + \beta \int_{\Omega} \hat{u}^2(x) \hat{u}^2(-x)\,\dd x = \int_\Omega k^2\abs{\nabla \hat{u}}^2 + \hat{u}^2 > 0
	\]
	From the third part of the proof of Lemma~\ref{lm:polarization_ineq}, we conclude that
	\[
		\int_{\Omega} \hat{u}_H^2(x) \hat{u}_{H}^2(-x)\,\dd x \leq \int_{\Omega} \hat{u}^2(x) \hat{u}^2(-x)\,\dd x,
	\]
	and hence that
	\[
		\int_{\Omega} \abs{\hat{u}^+_H}^4 + \beta \int_{\Omega} \hat{u}_H^2(x) \hat{u}_H^2(-x)\,\dd x >0.
	\]
	Thus, by Lemma~\ref{lm:nehari_rds}, there is a $\lambda$ such that $\lambda \hat{u}_H\in\hat{\mathcal{N}}_k$ and
	\[
		\hat{E}_k(\lambda u_H) = \sup_{t > 0} \hat{E}_k(t u_H).
	\]
	Write $m=\hat{E}_k(\hat{u})=\inf_{\hat{\mathcal{N}}_k}\hat E$. Since $(\lambda \hat{u})_H = \lambda u_H$, we have
	\[
		m\leq \hat{E}_k(\lambda \hat{u}_H)\leq \hat{E}_k(\lambda \hat{u}) \leq \hat{E}_k(\hat{u}) = m.
	\]
	By the uniqueness given by Lemma~\ref{lm:nehari_rds}, we have $\lambda = 1$. Hence, $\hat{u}_H \in \hat{\mathcal{N}}_k$.
\end{proof}
\begin{theorem}
	The minimizer $\hat{u}$ of $\hat{E}_k$ on $\hat{\mathcal{N}}_k$ is foliated Schwarz symmetry.
\end{theorem}
\begin{proof}
	By Lemmas~\ref{lm:polarization_ineq} and~\ref{lm:uh_still_in_Nk}, we see that, for every $H\in \mathcal{H}_0$, $u_H$ is still a classical solution of \eqref{eq:2coupled}. Then from Lemma~\ref{lm:fss}, we conclude that $u$ is foliated Schwarz symmetry with respect to some point.
\end{proof}
\begin{corollary}
	The minimizer $\hat{u}$ of $\hat{E}_k$ on $\hat{\mathcal{N}}_k$ has minimal period $2\pi$ and each component of $U=\Psi_k u$ has minimal period $2\pi/k$.
\end{corollary}
\section{Extensions}
\label{sec:ext}
In the preceding sections for clarity and simplicity of presentations we have focused on radially symmetric domains and solutions invariant under rotations symmetries. Our methods can be adapted for more general domains, as well as solutions invariant under more general group actions.
We list some here with the proofs omitted or just sketched.

\begin{remark}
	Checking through the proofs in Sections~\ref{sec:multiple_positive_solutions} and~\ref{sec:multiple_nodal_solutions} we see that,
	in Theorem~\ref{th:main} and~\ref{th:main_nod}, for dimension $n=3$, we only need $\Omega$ to be invariant under rotation symmetries with respect to some fixed axis, e.g., $\Omega$ can be cylinder-type or cone-type domains. We leave the statements to interested readers.
\end{remark}

\begin{remark}
The non-radial solutions we have constructed so far all are invariant under some rotation type symmetries. One would wonder whether this can be done for some group invariant functions under other type groups (as subgroups of $O(n)$). It turns out that this issue is more delicate than it might be thought of. We can do this for positive non-radial solutions but not sure about for nodal solutions.
We discuss this next, which was motivated by the work of \cite{weiNonradialSymmetricBound2007} where again within each symmetric class of functions they prove the existence of a ground state which is a non-radial positive solution. Our goal is to show the existence of an infinite sequence of non-radial positive solutions that share the same group invariance. We sketch a proof here.
\end{remark}

Let $\mathcal{G}$ be a nontrivial subgroup of $O(n)$. We denote by $\mathcal{G}x$ the set $\{gx\in\mathbb{R}^n\colon g\in\mathcal{G}\}$, the orbit of $x$ under $\mathcal{G}$. A function $u\colon \Omega\to\mathbb{R}$ is said to be $\mathcal{G}$-symmetric if $u=u\circ g$ for all $g\in\mathcal{G}$.
\begin{definition}\label{def:ap}	
	Let $\mathfrak{b}\in O(n)$, and let $\mathcal{G}$ be a nontrivial \emph{compact} subgroup of $O(n)$. We call the pair $(\mathcal{G}, \mathfrak{b})$ admissible if
	\begin{enumerate}[(a)]
		\item\label{enu:ap_a} $\mathfrak{b}$ is contained in the normalizer of $\mathcal{G}$, and $\mathfrak{b}^p\in\mathcal{G}$.
		\item\label{enu:ap_b} There exists $x_0\in\mathbb{R}^n\setminus\{0\}$ such that
			\[
				\mathcal{G}(\mathfrak{b}^s x_0) \cap \mathcal{G}(\mathfrak{b}^t x_0) = \varnothing \qquad \text{for $s\not\equiv t \pmod{p}$.}
			\]
	\end{enumerate}
\end{definition}
\begin{remark}
	In \cite{weiNonradialSymmetricBound2007}, Wei and Weth were the first to introduce the definition of admissible pair to study the existence of solution with invariance under group $\mathcal{G}$ of a coupled system in $\mathbb{R}^N$. Our assumptions are somewhat weaker than that in \cite{weiNonradialSymmetricBound2007}, but this is sufficient for our purpose.
\end{remark}

Condition~\ref{enu:ap_b} implies that $\mathfrak{b}, \mathfrak{b}^2, \dots, \mathfrak{b}^{p-1}\notin\mathcal{G}$ and $\mathfrak{b}^p\in\mathcal{G}$.
Condition~\ref{enu:ap_a} ensures that $\mathcal{G}_\mathfrak{b} = \mathcal{G}\cup \mathfrak{b}\mathcal{G}\cup\dots\cup \mathfrak{b}^{p-1}\mathcal{G}$ is a group and therefore makes it legitimate to define an action $\star$ of $\mathcal{G}_\mathfrak{b}$ on $(H^1_0(\Omega))^N$ as
\begin{gather*}
	\mathfrak{b} \star (u_1,\dots,u_N) = \sigma(u_1\circ \mathfrak{b}^{-1},\dots, u_N\circ \mathfrak{b}^{-1}),\\
	\mathfrak{g} \star (u_1,\dots,u_N) =  (u_1\circ \mathfrak{g}^{-1},\dots,u_N\circ \mathfrak{g}^{-1}) \qquad \text{for $\mathfrak{g}\in\mathcal{G}$}.
\end{gather*}
The fixed point space of $\star$ is precisely
\begin{align*}
	\mathcal{M}_{(\mathcal{G},\mathfrak{b})} = \{(u_1,\dots, u_N)\in(H^1_0(\Omega))^N \colon
	&\text{$u_i$ is $\mathcal{G}$-symmetric for $i=1,2,\dots,N$} \\
	&\text{and $u_{j+1} = u_{j}\circ \mathfrak{b}^{-1}$ for $p\nmid j$}\}.
\end{align*}
We proceed by defining the Nehari-type manifold in $\mathcal{M}_{(\mathcal{G},\mathfrak{b})}$ by
\begin{gather*}
	\mathcal{N}_{(\mathcal{G},\mathfrak{b})} = \{U=(u_1,\dots, u_N)\in \mathcal{M}_{(\mathcal{G},\mathfrak{b})} \colon u_j\neq 0,\partial_j E^+ (U)u_j=0,\forall j=1,2,\dots,N\}.
\end{gather*}
\begin{example}
	\begin{enumerate}
		\item Let $k$ be a positive integer and let $\mathfrak{b} = R_{2\pi/(pk)}$ (we recall that $R_{\theta}$ is defined by~\eqref{eq:R_theta}).
		Put $\mathcal{G}_0 = \{\id, \mathfrak{b}^p, \dots, \mathfrak{b}^{(k-1)p}\}$. Then $\mathcal{G}_0$ is a finite subgroup of $O(n)$ and $(\mathcal{G}_0,\mathfrak{b})$ clearly satisfies the admissibility condition~\ref{enu:ap_a}.
		Taking $x_0=(1,0,0)$, one can easily check that~\ref{enu:ap_b} also holds.
		In this case, $\mathcal{M}_{(\mathcal{G}_0,\mathfrak{b})}$ is precisely the subspace $\mathcal{M}_k^+$ we have defined in Section~\ref{sec:multiple_positive_solutions}.
		\item We use the notation from the previous example. Let $\mathcal{G}$ be the group generated by the reflection $F\colon (x,y,z)\mapsto(x,y,-z)$ and elements in $\mathcal{G}_0$. Since $\mathfrak{b}$ commutes with $F$, and $\mathcal{G}x_0 = \mathcal{G}_0 x_0$, $(\mathcal{G},\mathfrak{b})$ is also an admissible pair.
		\item Let $n=3$ and $N=2$. Set $\mathcal{G} = \{R_\theta\colon \theta\in[0,2\pi)\}$ and let $\mathfrak{b}$ be the reflection $(x,y,z) \mapsto (x,y,-z)$. Since $\mathfrak{b}R_\theta = R_{\theta}\mathfrak{b}$, the admissibility condition~\ref{enu:ap_a} is satisfied for $(\mathcal{G},\mathfrak{b})$. Taking $x_0 = (0, 0, 1)$, we have $\mathcal{G}x_0 = \{(0,0,1)\}$, $\mathcal{G}(\mathfrak{b}x_0) = \{(0,0,-1)\}$ and hence $\mathcal{G}x_0\cap \mathcal{G}(\mathfrak{b}x_0) = \varnothing$. Therefore, $(\mathcal{G},\mathfrak{b})$ is an admissible pair.
		\item (An example from~\cite{weiNonradialSymmetricBound2007}) Let $n=3$ and $N=2$. Consider the tetrahedral group $\mathcal{G}$ generated by the coordinate permutations $(x_1,x_2,x_3)\mapsto(x_{\pi_1}, x_{\pi_2}, x_{\pi_3})$ and the map $(x_1,x_2,x_3)\mapsto (x_1,-x_2,-x_3)$. Let $\mathfrak{b}$ be the reflection $x\mapsto -x$. Then $\mathfrak{b}$ commutes with each elements of $\mathcal{G}$. Taking $x_0 = (1,1,1)$ we have
			\[
				\mathcal{G}x_0 = \{(1,1,1), (-1,-1,1), (1,-1,-1), (-1,1,-1)\},
			\]
			and
			\[
				\mathcal{G}(\mathfrak{b}x_0) = \{(-1,-1,-1), (1,1,-1), (-1,1,1), (1,-1,1)\}.
			\]
		Hence $\mathcal{G}x_0\cap \mathcal{G}(\mathfrak{b}x_0) = \varnothing$ and $(\mathcal{G},\mathfrak{b})$ is an admissible pair.
	\end{enumerate}
\end{example}

Similarly to that in Section~\ref{sec:multiple_positive_solutions}, we have the following lemma.
\begin{lemma}
	The subspace $\mathcal{M}_{(\mathcal{G}, \mathfrak{b})}$ and the submanifold $\mathcal{N}_{(\mathcal{G}, \mathfrak{b})}$ are natural constraints, i.e., every constrained critical point of $E$ on them is also a critical point of $E$.
\end{lemma}
Likewise, we only need to prove the following proposition.
\begin{proposition}
	Let $(\mathcal{G}, \mathfrak{b})$ be an admissible pair. Denote by $\mathbb{S}^{2m-1}$ the unit sphere in $\mathbb{C}^m$. For each positive integer  $m$, there exists a continuous map $\psi\colon \mathbb{S}^{2m-1} \to\mathcal{N}_{(\mathcal{G}, \mathfrak{b})}$, such that
	\[
		\psi( e^{2\pi i/p} z) = \sigma\psi(z).
	\]
\end{proposition}
\begin{proof}
	For each $i \in \{1,\dots, m\}$, $j\in\{1, \dots, B\}$ and $\theta\in \mathbb{R}/(2\pi\mathbb{Z})$, one can find $U^{i,j}_\theta\in C_0^\infty(\Omega)\setminus\{0\}$ satisfying the following conditions.
	\begin{enumerate}[(U1$''$)]
		\item\label{enu:u1_pp} The map $\mathbb{R}/(2\pi\mathbb{Z})\to C_0^\infty(\Omega)\colon t\mapsto U_t^{i,j}$ is continuous and $U_t^{i,j}$ is $\mathcal{G}$-symmetric.
		\item\label{enu:u2_pp} $\supp U^{i_1,j_1}_{\theta_1}\cap \supp U^{i_2,j_2}_{\theta_2} = \varnothing$ for all $\theta_1,\theta_2\in\mathbb{R}/(2\pi\mathbb{Z})$ and $(i_1,j_1)\neq(i_2,j_2)$.
		\item\label{enu:u3_pp} $\supp U^{i,j}_{\theta} \cap \supp U^{i,j}_{\theta+\frac{2\pi s}{p}} = \varnothing$ and $U^{i,j}_{\theta + \frac{2\pi s}{p}} = U^{i,j}_{\theta}\circ (\mathfrak{b}^{-1})^s$ if $\theta\in\mathbb{R}/2\pi \mathbb{Z}$ and $s \in \{1,\dots,p-1\}$.
	\end{enumerate}
	To be more precise, for each $i \in \{1,\dots, m\}$, $j\in\{1, \dots, B\}$ and $\ell\in\{1,2,\dots,p\}$, choose radially symmetric domains $\Omega_{i,j},\hat \Omega_{i,j}\subset\Omega$ such that all the $\Omega_{i,j}$-s and $\hat\Omega_{i,j}'$-s are pairwise disjoint.
	Let $x_0$ be given by \ref{enu:ap_b} in Definition~\ref{def:ap}. Choose $\lambda^{i,j}, \hat\lambda^{i,j}\in \mathbb{R}$ such that $\lambda^{i,j} x_0\in \Omega_{i,j}$ and $\hat\lambda^{i,j} x_0\in \hat\Omega_{i,j}$.
	Since $\mathcal{G}\subset O(n)$, $\mathfrak{b}\in O(n)$, it follows that $\mathcal{G}(\mathfrak{b}^s \lambda^{i,j}x_0) \subset \Omega^{i,j}$ and $\mathcal{G}(\mathfrak{b}^s \hat\lambda^{i,j}x_0) \subset \hat\Omega^{i,j}$ for $s=0,1,\dots,p$.
	The set $\mathcal{G}x_0$ is compact, and, in consequence, there exists a $\delta>0$ such that
	\begin{equation}\label{eq:obit_disjoint}
		\begin{split}
			N_{\delta}( \mathcal{G}(\mathfrak{b}^s \lambda^{i,j}x_0)) \cap N_{\delta}(\mathcal{G}(\mathfrak{b}^t \lambda^{i,j}x_0)) = \varnothing,\\
			\text{and} ~ N_{\delta}( \mathcal{G}(\mathfrak{b}^s \hat\lambda^{i,j}x_0)) \cap N_{\delta}(\mathcal{G}(\mathfrak{b}^t \hat\lambda^{i,j}x_0)) = \varnothing
		\end{split}
	\end{equation}
	for $s\not\equiv t \pmod{p}$ and all $i,j$,
	by the assumption~\ref{enu:ap_b} of admissible pair in Definition~\ref{def:ap}. Here and subsequently,
	\[
		N_{\delta}(A) = \{x\in\mathbb{R}^n\colon d(x, A)<\delta\} \qquad \text{for $A\subset \mathbb{R}^n$}.
	\]
	We can also choose $\delta$ sufficiently small such that
	\[
		\text{ $N_{\delta}(\mathcal{G}(\mathfrak{b}^s \lambda^{i,j}x_0)) \subset \Omega_{i,j}$ and $N_{\delta}(\mathcal{G}(\mathfrak{b}^s \hat\lambda^{i,j}x_0)) \subset \hat\Omega_{i,j}$.}
	\]
	Then we can find a nonzero function $\varphi^{i,j}\in H^1_0(N_\delta(\mathcal{G}\lambda^{i,j}x_0))$ and $\hat\varphi^{i,j}\in H^1_0(N_\delta(\mathcal{G}\hat\lambda^{i,j}x_0))$
	such that $\varphi^{i,j}$ and $\hat\varphi^{i,j}$ are $\mathcal{G}$-symmetric. Indeed, let $\varphi_0\in C_0^\infty(B_\delta(\lambda^{i,j} x_0))$ and let
	\[
		\varphi^{i,j}(x) = \sup_{\mathfrak{g}\in \mathcal{G}} \varphi_0(\mathfrak{g}^{-1} x) \qquad \forall x\in\Omega.
	\]
	Since the upper envelope (pointwise maximum) of functions preserves the property of being Lipschitz, $\varphi^{i,j}$ is at least Lipschitz and hence in $H^1_0(N_\delta(\mathcal{G}\lambda^{i,j}x_0))$. The construction for $\hat\varphi^{i,j}$ is same.

	Let $\eta \in C_0^\infty(\mathbb{R}/(2\pi\mathbb{Z}))$ be a function such that
	\[
		\eta>0 \text{~in} \left(-\frac{\pi}{p}, \frac{\pi}{p}\right)+2\pi\mathbb{Z} \subset \mathbb{R}/(2\pi\mathbb{Z}),
	\]
	and vanishes outside. Set
	\[
		\eta_{\iota}(\theta) = \eta\left(\theta-\frac{\iota\pi}{p}\right)
	\]
	for each $\iota\in\{1,2,\dots,2p\}$. Thus,
	\[
		\supp\eta_1 = \left[0, \frac{2\pi}{p}\right]+2\pi\mathbb{Z},~\supp\eta_2 = \left[\frac{\pi}{p}, \frac{3\pi}{p}\right]+2\pi\mathbb{Z},\dots,~\supp\eta_{2p} = \left[-\frac{\pi}{p}, \frac{\pi}{p}\right]+2\pi\mathbb{Z}.
	\]

	Then we set $\varphi_{\ell}^{i,j} = \varphi^{i,j}\circ (\mathfrak{b}^{-1})^\ell$, $\hat\varphi_{\ell}^{i,j} = \hat\varphi^{i,j}\circ (\mathfrak{b}^{-1})^\ell$ and
	\[
		U_\theta^{i,j} = \sum_{\ell=1}^{p} \eta_{2\ell-1}(\theta) \varphi_\ell^{i,j} + \eta_{2\ell}(\theta) \hat\varphi_\ell^{i,j}.
	\]
	Thus, $U_\theta^{i,j}$ satisfies \ref{enu:u1_pp}-\ref{enu:u3_pp}. The validity of \ref{enu:u1_pp} and \ref{enu:u2_pp} are obvious, while \ref{enu:u3_pp} follows from our choice of $\eta_\ell$ and $\varphi_{\ell}^{i,j}$.
	In fact, when $\theta\in (\frac{t\pi}{p}, \frac{(t+1)\pi}{p})$ for some $t\in\{0,1,\dots,2p-1\}$, $\eta_\iota\neq 0$ if and only if $\iota \equiv t$ or $t+1 \pmod{2p}$.
	With setting $\eta_{\iota+2p} = \eta_{\iota}$, $\varphi_{\ell+p}=\varphi_\ell$ and $\hat\varphi_{\ell+p}=\hat\varphi_\ell$, we conclude that
	\begin{equation*}
		U^{i,j}_{\theta} = \eta_t(\theta)\varphi^{i,j}_{\ell_1} + \eta_{t+1}(\theta)\hat\varphi^{i,j}_{\ell_2} \neq 0,
	\end{equation*}
	where $\ell_1 = \lceil t/2 \rceil$ and $\ell_2 =  \lceil (t+1)/2 \rceil$.
	Similarly, for $s\in\{1,2,\dots,p-1\}$, we conclude that
	\begin{align*}
		U^{i,j}_{\theta+\frac{2\pi s}{p}} &=\eta_{t+2s}\left(\theta+\frac{2\pi s}{p}\right)\varphi^{i,j}_{\ell_1} +\eta_{t+2s+1}\left(\theta+\frac{2\pi s}{p}\right)\hat\varphi^{i,j}_{\ell_2} ,\\
		&=\eta_t(\theta)\varphi^{i,j}_{\ell_1+s} + \eta_{t+1}(\theta)\hat\varphi^{i,j}_{\ell_2+s}\\
		&= \left(\eta_t(\theta)\varphi^{i,j}_{\ell_1} + \eta_{t+1}(\theta)\hat\varphi^{i,j}_{\ell_2}\right)\circ (\mathfrak{b}^{-1})^s.
	\end{align*}
	From~\eqref{eq:obit_disjoint}, it follows that
	\[
		\supp U^{i,j}_{\theta} \cap \supp U^{i,j}_{\theta+\frac{2\pi s}{p}} = \varnothing.
	\]
	The case that $\theta = \frac{t\pi}{p}$ for some $t$ is similar, and we omit it.

	Hence, we can define $\psi_0\colon \mathbb{S}^{2m-1}\to (H_0^1(\Omega))^N$
	by
	\begin{align*}
		&{} \psi_0(z)
		= \psi_0(r_1\mathrm{e}^{\mathrm{i}\theta_1}, \dots, r_m\mathrm{e}^{\mathrm{i}\theta_m})\\
		=&{} \left(
			U^{(1)}_1, \dots, U^{(1)}_p;\quad
			\dots;\quad
			U^{(B)}_1, \dots, U^{(B)}_p
		\right)
	\end{align*}
	where
	\[
		U^{(j)}_\ell = \sum_{i = 1}^{m} r_i U_{\theta_i + \frac{2\pi \ell}{p}}^{i,j} , \qquad \text{for $\ell=1,\dots,p$, $j=1,\dots, B$.}
	\]
	We thus have
	\begin{align*}
		\psi_0(\mathrm{e}^{2\pi\mathrm{i}/p} z) &= \Phi(r_1\mathrm{e}^{\mathrm{i}(\theta_1+2\pi\mathrm{i}/p)}, \dots, r_m\mathrm{e}^{\mathrm{i}(\theta_m+2\pi\mathrm{i}/p)})\\
		&= \sigma\psi_0(z).
	\end{align*}

	From \ref{enu:u1_pp}-\ref{enu:u3_pp}, it follows that $\psi_0(z)\in\mathcal{M}_{(\mathcal{G},\mathfrak{b})}$, and that the supports of all components of $\psi_0(z)$ are pairwise disjoint.
	Since $r_i$-s are not all zero for $(r_1\mathrm{e}^{\mathrm{i}\theta_1}, \dots, r_m\mathrm{e}^{\mathrm{i}\theta_m}) \in \mathbb{S}^{2m-1}$, each component of $\psi_0(z)$ is nonzero by the definition. These facts allow one to construct a $C^1$ map $\Lambda \colon \psi_0(\mathbb{S}^{2m-1})\to\mathcal{N}_{(\mathcal{G},\mathfrak{b})}$ such that $\Lambda\sigma = \sigma\Lambda$.
	Then the rest of the proof runs as before.
\end{proof}
With these preparations and using the framework in Section 2, we can prove the following theorem.
\begin{theorem}
	Let $(\mathcal{G},\mathfrak{b})$ be an admissible pair. Then under the assumptions \ref{itm:as_a}-\ref{itm:as_d}, Problem~\eqref{eq:main} admits infinitely many positive solutions of the form $(u_1,\dots,u_N)$ that $u_i$ is $\mathcal{G}$-symmetric but not $\mathfrak{b}$-symmetric for $i=1,2,\dots,N$ and $u_{j+1} = u_{j}\circ \mathfrak{b}^{-1}$ for $p\nmid j$. In particular, each $u_i$ is non-radial.
\end{theorem}
Using the third example in Example 1, we have the following result.
\begin{corollary}\label{coro:2d} Assume $n=3$ and $N=2$.
    Then under the assumptions \ref{itm:as_a}-\ref{itm:as_d}, Problem~\eqref{eq:main} admits an unbounded sequence $S = \{(u_{1,l},\cdots,u_{N,l})\colon l\in\mathbb{N}\}$ of positive solutions such that each $u_{i,l}$ is radial in $(x_1,x_2)$ and even in $x_3$, but is not radial in $(x_1,x_2,x_3)$.
\end{corollary}

\noindent{\bf Data availability statement.} Data sharing not applicable to this article as no datasets were
generated or analysed during the current study.
\bibliography{reference.bib}
\end{document}